\documentclass[12pt]{amsart}
\usepackage{amsthm}
\usepackage{amsmath}
\usepackage{amsfonts}
\usepackage{amssymb}

\makeatletter
\@namedef{subjclassname@2010}{%
  \textup{2010} Mathematics Subject Classification}

\makeatother

\begin{document}
 \baselineskip=17pt
\begin{abstract}{In this paper we prove that  some well-known  K\"ahler surfaces with zero scalar curvature are   QCH K\"ahler.   We prove that the family of generalized Taub-NUT K\"ahler surfaces parameterized by $k\in[-1,1]$ is   of   orthotoric type for $k\in(-1,1)$ and of Calabi type for $k\in\{-1,1\}$ and the Burns metric is of  Calabi type.    }\end{abstract}


   \title{Some QCH K\"ahler surfaces with zero scalar curvature.}

 \author[W\l odzimierz Jelonek]{W{\l}odzimierz Jelonek }
\address{Institute of Mathematics\\
Cracow University of Technology\\
Warszawska 24\\
31-155      Krak\`{o}w,  POLAND.}

 \email{  wjelon@pk.edu.pl}
\maketitle

\newtheorem{thm}{Theorem}[section]

\newtheorem{dfn}{Definition}[section]

\newtheorem{rem}{Remark}[section]

\newtheorem{lm}{Lemma}[section]

\newtheorem{prop}{Proposition}[section]

\newtheorem{cor}{Corollary}[section]

\newcommand\G{\Gamma}
\newcommand\bJ{\overline{J}}
\newcommand\alp{\alpha}
\newcommand\e{\epsilon}
\newcommand\n{\nabla}
\newcommand\om{\omega}
\newcommand\Om{\Omega}
\newcommand\rt{\rightarrow}
\newcommand\w{\wedge}
\newcommand\bp{\overline{\partial}}
\newcommand\De{\mathcal D}
\newcommand\p{\partial}
\newcommand\al{\alpha}
\newcommand\be{\beta}
\newcommand\g{\gamma}
\newcommand\lb{\lambda}
\newcommand\E{\mathcal E}
\newcommand\DE{\mathcal D^{\perp}}

\newcommand\0{\Omega}
\newcommand\bt{\beta}

\newcommand\s{\sigma}
\newcommand\J{\Cal J}

\newcommand\de{\delta}
\newcommand\R{\mathcal R}
\newcommand\dl{\delta}
\newcommand\ka{\kappa}
\newcommand\m{(M,g,J)}

 \section{Introduction} In the paper we study  toric  K\"ahler surfaces with zero scalar curvature  (see [W]).  K\"ahler surfaces are an object of intense study  (see  [D], [Do],  [D-T], [G-R], [N], [W]).
We show that  scalar flat K\"ahler surfaces studied by Weber in [W] are   QCH  K\"ahler surfaces and  the Burns metric is a Calabi type QCH K\"ahler surface.  In general   toric  K\"ahler manifold $(M,g,J)$ has nondegenerate
Weyl tensor  $W^-$, i.e.  it has three different eigenvalues.    QCH K\"ahler surfaces have degenerate Weyl tensor  $W^-$, it has only two eigenvalues.  What is more the almost Hermitian structure $I$ whose K\"ahler form is an
eigenvector of $W^-$ corresponding to the simple eigenvalue satisfies $\rho(I.,I.)=\rho(.,.)$, where $\rho$ is the Ricci tensor of $(M,g,J)$.

 {\it  QCH K\"ahler surfaces}   (i.e. K\"ahler surfaces with quasi constant holomorphic sectional curvature) (see [G-M], [J-1]-[J-8], [J-M], [M]) are K\"ahler surfaces  $(M, g, J)$
admitting a global, $2$-dimen\-sional, $J$-invariant distribution
$\De$ having the following property: The holomorphic curvature
$$K(\pi)=R(X, J X, J X, X)$$ of any $J$-invariant $2$-plane $\pi\subset
T_xM$, where $X\in \pi$ and $g(X, X)=1$, depends only on the point
$x$ and  number $|X_{\De}|=\sqrt{g(X_{\De},X_{\De})}$, where
$X_{\De}=p_{\De}X$ is the orthogonal projection of $X$ onto $\De$  (see [J-1]).  Every  QCH K\"ahler surface admits an opposite almost Hermitian structure $I$, (i.e. its K\"ahler form $\om_I$ is anti-selfdual), such that the Ricci tensor $\rho$ of $(M,g,J)$ is $I$-invariant and $(M,g,I)$ satisfies the second Gray condition  (see [G], p.605):

\begin{equation*}(G2) R(X, Y, Z, W)-R(IX, IY, Z, W)=\end{equation*}\begin{equation*}R(IX, Y, IZ, W)+R(IX, Y, Z, IW).\end{equation*}
 If $R$ is the curvature tensor of a
QCH K\"ahler manifold $\m$, then there exist functions $a,b,c\in
C^{\infty}(M)$ such that
\begin{equation}R=a\Pi+b\Phi+c\Psi,\end{equation} where $\Pi$ is the
standard K\"ahler tensor of constant holomorphic curvature i.e.
\begin{equation}
\Pi(X,Y,Z,U)=\frac14(g(Y,Z)g(X,U)-g(X,Z)g(Y,U)\end{equation}
\begin{equation*}+g(JY,Z)g(JX,U)-g(JX,Z)g(JY,U)-2g(JX,Y)g(JZ,U)),\end{equation*}
the tensor $\Phi$ is defined by the following relation

\begin{equation} \Phi(X,Y,Z,U)=\frac18(g(Y,Z)h(X,U)-g(X,Z)h(Y,U)\end{equation}
\begin{equation*}+g(X,U)h(Y,Z)-g(Y,U)h(X,Z)
+g(JY,Z)h(JX,U)\end{equation*}\begin{equation*}-g(JX,Z)h(JY,U)+g(JX,U)h(JY,Z)-g(JY,U)h(JX,Z)\end{equation*}
\begin{equation*}-2g(JX,Y)h(JZ,U)-2g(JZ,U)h(JX,Y))\end{equation*}and finally
\begin{equation}\Psi(X,Y,Z,U)=-h(JX,Y)h(JZ,U)=-(h_J\otimes h_J)(X,Y,Z,U).\end{equation} where $h_J(X,Y)=h(JX,Y)$ and   $h=g\circ(p_{\De}\times p_{\De})$.
In [J-1] we have proved (th.3.1 in [J-1]):
\begin{thm} {\it Let $\m$ be a K\"ahler surface.  If $\m$ is a QCH
manifold, then $W^-=c(\frac16\Pi-\Phi+\Psi)$ and $W^-$ is
degenerate.  The 2-form $\overline{\om}$ is an eigenvector of
$W^-$ corresponding to a simple eigenvalue of $W^-$ and $
\overline{J}$ preserves the Ricci tensor. On the other hand let us
assume that $\m$ admits an opposite almost complex structure $
\overline{J}$ such that $Ric(\overline{J},\overline{J})=Ric$. Let
$
\mathcal E=\operatorname{ker}(J\overline{J}-Id),\De=\operatorname{ker}(J\overline{J}+Id)$. If
$W^-=\frac{\kappa}2(\frac16\Pi-\Phi+\Psi)$ or equivalently if the
half-Weyl tensor $W^-$ is degenerate and $ \overline{\om}$ is an
eigenvector of $W^-$ corresponding to a simple eigenvalue of $W^-$
then   $\m$ is a QCH manifold.}\end{thm}

We also have
\begin{equation*}R=(\frac{\tau}6-\delta+\frac\kappa{12})\Pi+(2\delta-\frac{\kappa}2)\Phi+\frac{\kappa}2\Psi.\end{equation*}

Every QCH K\"{a}hler surface is a holomorphically
pseudosymmetric K\"{a}hler manifold  (see [O], [J-7]).  In fact
 in the
case of QCH K\"{a}hler surfaces we have
\begin{equation}  R.R=\frac16(\tau-\kappa)\Pi.R\end{equation}
where $\tau$ is the scalar curvature and $\kappa$ is the
conformal scalar curvature of $(M,g, \overline{J})$.   In fact a K\"ahler surface is a QCH K\"ahler surface if it admits an opposite almost Hermitian structure $I$ such that $\rho(I.,I.)=\rho(.,.)$, the  Weyl tensor $W^-$ is degenerate and the K\"ahler form $\om_I$ of the structure $I$ is a simple eigenvector of $W^-$.    If   $I$ is an opposite Hermitian structure and $\rho(I.,I.)=\rho(.,.)$, then  the second condition is also satisfied and $(M,g,J)$ is a QCH K\"ahler surface  (see [A-G], [J-1]).

Here we restrict to  QCH K\"ahler surfaces for which the structure $I$ is integrable.   These are essentially generalized Calabi type K\"ahler surfaces and generalized orthotoric type K\"ahler surfaces.  If additionally  $I$ is conformally K\"ahler, then  they are  Calabi type and orthotoric type K\"ahler surfaces.  The Lee form $\theta$ of  the Hermitian structure  $I$ satisfies  $d\om_I=2\theta\w\om_I$  where $\om_I(X,Y)=g(IX,Y)$ is the K\"ahler form of  $(M,g,I)$.

\begin{dfn} A generalized Calabi type QCH K\"ahler surface  is a QCH K\"ahler surface for which  $I$ is Hermitian, $|\n I|>0$  and one of the distributions $\mathcal D, \mathcal D^{\perp}$ coincides with $\mathcal D_0=\operatorname{span}\{\theta^\sharp,I\theta^\sharp\}=\operatorname{Null}I=\{X:\n_XI=0\}$ and consequently is is integrable (see [J-4],[J-6]).\end{dfn}

 \begin{dfn}A generalized orthotoric QCH  K\"ahler surface is a QCH K\"ahler surface with integrable structure $I$ with $|\n I|>0$ and such that  both  distributions  $\mathcal D, \mathcal D^{\perp}$ satisfy  $\mathcal D\cap\mathcal D_0=\mathcal D^{\perp}\cap\mathcal D_0=\{0\}$.\end{dfn}

An  orthotoric type K\"ahler surface is  a K\"ahler surface $(M,g,J)$  with a metric
\begin{align*}g=(\xi-\eta)(\frac 1{F(\xi)}d\xi^2-\frac1{G(\eta)}d\eta^2)+\frac1{\xi-\eta}(F(\xi)(dt+\eta dz)^2\\-G(\eta)(dt+\xi dz)^2)\end{align*} (see [J-1], [J-7] and references there).
These are K\"ahler surfaces with zero scalar curvature if and only if   $F(\xi)=A\xi^2+a\xi+b,  G(\eta)=A\eta^2+c\eta+d$.  Note that every orthotoric K\"ahler surface is a QCH K\"ahler surface (see [J-1]).

A QCH  K\"ahler surface with integrable $I$ is of generalized Calabi type if and only if $d(\om_J-\om_I)=\phi\w(\om_J-\om_I)$  or  $d(\om_J+\om_I)=\phi\w(\om_J+\om_I)$ for some 1-form $\phi$.
A generalized Calabi (orthotoric) K\"ahler surface is called of Calabi (orthotoric) type if the Hermitian structure $I$ is locally conformally K\"ahler. Note that the generalized Calabi type K\"ahler surface with zero scalar curvature has to be of Calabi type which follows from the classification in [J-M], [M].

\begin{dfn}  We say also that  a QCH K\"ahler surface $(M,g,J)$ is of orthotoric or of  Calabi type  if there exists an open and dense subset $U\subset M$  such that  $(U,g_{|U},J_{|U})$ is of orthotoric or of Calabi type K\"ahler surface.\end{dfn}

In the paper we   prove  that all K\"ahler surfaces investigated by Weber in [W] are  QCH K\"ahler surfaces of  orthotoric or Calabi type.   We also prove that the Burns metric  (see [L], [B])  is of Calabi type. The generalized Taub-NUT K\"ahler surfaces are orthotoric for $k\in (-1,1)$  and Calabi type for $k\in\{-1,1\}$.   The exceptional half plane surface is of Calabi type.    If $(M,g,J)$ is not hyperk\"ahler, then the opposite Hermitian structure  $I$ has the K\"ahler form proportional to the Ricci form $\rho$ of $(M,g,J)$,  which   explains why   $\rho$ is an eigenvalue of the Weyl tensor $W^-$   and why $W^-$  is degenerate,  which was noticed by Weber in [W].

\section{ Toric K\"ahler surfaces.}  We assume that $(M,g,J)$ is a K\"ahler surface admitting two linearly independent holomorphic Killing commuting vector fields  $X_1=X,X_2=Y$, i.e.  $\mathfrak{h}=\operatorname{span} \{X_1,X_2\}\subset \mathfrak{hol}(M)$.  Let  $\phi_1,\phi_2$  be potentials of the fields $X_1,X_2$ i.e. $X_i=J\nabla\phi_i$. The moment map   $(\phi_1,\phi_2):M\rightarrow \Sigma\subset\mathbb{R}^2$ produces the Riemannian projection onto the metric polytope.   Let  $\frak{h}\subset \frak{hol}(M)$   be a subalgebra of  holomorphic Killing vector fields of  $(M,g,J)$ $\frak{h}=\operatorname{span}\{X,Y\}$  where  $[X,Y]=0$.   We have $[L_X,i_Y]=i_{[X,Y]}=0$  which means   $L_Xi_Y=i_YL_X$. Hence  $d\om(X,Y)=di_Xi_Y\om=L_X(i_Y\om)-i_Xdi_Y\om=i_YL_X\om-i_XL_Y\om=0$.  Consequently  $\om(X,Y)=const$.   Now we assume that $\om(X,Y)=0$.  Assume further that $X=J\n\phi_1, Y=\n\phi_2$.   Since  $[X,JY]=J[X,Y]=0$  and $\n\phi_i$ are also holomorphic, it follows  $0=[\n\phi_1,J\n\phi_2]=J[\n\phi_1,\n\phi_2]=0$ .   It means that all vectors field $\n\phi_1,\n\phi_2,J\n\phi_1,J\n\phi_2$ commute.   Let  $\phi^1_{t},\phi^2_{t},\psi^1_{t},\psi^2_{t}$   be one-parameter diffeomorphism groups generated by  $\n\phi_1,\n\phi_2,J\n\phi_1,J\n\phi_2$.   Then the map
$(t_1,t_2,t_3,t_4)\rightarrow \phi^1_{t_1}\circ\phi^2_{t_2}\circ\psi^1_{t_3}\circ\psi^2_{t_4}$   is a parametrization of $M$  and  $\frac{\p}{\p t_1}=\n\phi_1,\frac{\p}{\p t_2}=\n\phi_2,\frac{\p}{\p t_3}=X,\frac{\p}{\p t_4}=Y$.    Note that the distributions  $\operatorname{span}\{\n\phi_1,\n\phi_2\}$,   $\operatorname{span}\{X,Y\}$ are integrable and orthogonal.  Since $X\phi_1=X\phi_2=Y\phi_1=Y\phi_2=0$,  it follows that $\phi_1,\phi_2, t_3,t_4$  is a coordinate system of  $M$   and $\frac{\p}{\p\phi_i}\in\{\n\phi_1,\n\phi_2\}$.   Hence  $\n\phi_1=g^{11}\frac{\p}{\p\phi_1}+g^{12}\frac{\p}{\p\phi_2},
\n\phi_2=g^{12}\frac{\p}{\p\phi_1}+g^{22}\frac{\p}{\p\phi_2}$.   Note that    $g(\n\phi_1,\n\phi_1)=(g^{11})^2g_{11}+2g^{11}g^{12}g_{12}+(g^{22})^2g_{22}=g^{11}$  and  $$g(\n\phi_1,\n\phi_2)=g^{12}, g(\n\phi_2,\n\phi_2)=g^{22}.$$ Let    $$[G^{ij}]=\begin{pmatrix} g(X,X)&g(X,Y)\\g(X,Y)&g(Y,Y)\end{pmatrix}.$$   Then    $[G^{ij}]=[g^{ij}]$  and the metric is
$$g=G_{ij}d\phi_id\phi_j+G^{ij}d\theta_id\theta_j$$   where  $\theta_1=t_3,\theta_2=t_4$.  We write  $g_{\Sigma}=G_{ij}d\phi_id\phi_j$.  On the toric K\"ahler surfaces with zero scalar curvature the volumetric normal coordinates $(x,y)$ on $\sigma$ are defined.  We have  $x=\sqrt{|X_1|^2|X_2|^2-<X_1,X_2>^2}$  which is a harmonic function on $(\Sigma,g_{\Sigma})$ and  $y=-*dx$ (for the choice of orientation $*$ see [W]).

Note that   $\om=d\phi_1\w d\theta_1+d\phi_2\w d\theta_2$ is a K\"ahler form, since   $g(J\n\phi_i,\frac{\p}{\p\theta_j})=g(X_i, X_j)=g^{ij}$.

\section{ Generalized Taub-NUT surfaces}   The generalized Taub-NUT surfaces were first described by Donaldson in [Do], (see also [W]). For the generalized Taub-NUT surfaces   we  have $M=\mathbb C^2$ and  (see [W])

\begin{align*}\psi_1=\frac1{\sqrt{2}}(-y+\sqrt{x^2+y^2})+\frac{\alpha}2x^2,\\\psi_2=\frac1{\sqrt{2}}(y+\sqrt{x^2+y^2})+\frac{\beta}2x^2\end{align*}

and

\begin{equation*}  g_{\Sigma}=\frac{1+2M(ky+\sqrt{x^2+y^2})}{\sqrt{x^2+y^2}}(dx\otimes dx+dy\otimes dy),\end{equation*}
where $x,y$ are volumetric coordinates on $\Sigma$,  $k=\frac{\alpha-\beta}{\alpha+\beta}, M=\frac{\alpha+\beta}{2\sqrt{2}}, \alpha\ge0,\beta\ge0$.

Note that for $k=0$ we get  standard Taub-NUT  hyperk\"ahler surface.

\begin{thm}  Let  $(M,g,J)$  be a generalized Taub-NUT K\"ahler surface  with $k\in(-1,1)$.  Then   $(M,g,J)$  is an orthotoric type QCH  K\"ahler surface with a metric (on an open and dense subset) \begin{equation*}g=(\xi-\eta)(\frac 1{F(\xi)}d\xi^2-\frac1{G(\eta)}d\eta^2)+\end{equation*}\begin{equation*}\frac1{\xi-\eta}(F(\xi)(dt+\eta dz)^2-G(\eta)(dt+\xi dz)^2),\end{equation*} where   $F(\xi)=a\xi+b,  G(\eta)=c\eta+d$,  and  $a,c>0$, $da-bc>0$.   We also have $M=\frac{(c+a)a^2c^2}{4(da-bc)^2}$, $k=\frac{a-c}{a+c}$.\end{thm}

\begin{proof}Let $F(\xi)=a\xi+b,  G(\eta)=c\eta+d$  where   $a,c>0$  and  $\xi\ge \xi_0=-\frac ba,  \eta\le\eta_0=-\frac dc$. Assume that  $-\frac ba+\frac dc>0$,  which means  $da-bc>0$.   Consider  a K\"ahler surface $(M,g,J)$  with a metric
\begin{equation*}g=(\xi-\eta)(\frac 1{F(\xi)}d\xi^2-\frac1{G(\eta)}d\eta^2)+\frac1{\xi-\eta}(F(\xi)(dt+\eta dz)^2-G(\eta)(dt+\xi dz)^2).\end{equation*}

  The fields  $X_1=\frac{\p}{\p t}, X_2=\frac{\p}{\p z}$  are holomorphic Killing with potentials   $\phi_1=\xi+\eta, \phi_2= \xi\eta$  respectively  (see [A-C-G]).  Note that $g(X_1,X_1)=\frac1{\xi-\eta}(F-G),
g(X_2,X_2)=\frac1{\xi-\eta}(F\eta^2-G\xi^2),   g(X_1,X_2)=\frac1{\xi-\eta}(\eta F-\xi G)$.

   Let $\widetilde{x}=\frac2a\sqrt{a\xi+b},  \widetilde{y}=-\frac2c\sqrt{-(c\eta+d)}$  be isothermal coordinates on  $\Sigma$,  $\widetilde{z}=\widetilde{x}+i\widetilde{y}$.  Then the metric on $\Sigma$ is   $g_{\Sigma}=(\xi-\eta)(d\widetilde{x}^2+d\widetilde{y}^2)$.

  The area of a parallelogram defined by $X_1,X_2$  is   $\mathcal V=\sqrt{-FG}$.  Let us introduce the volumetric coordinates  $x=\sqrt{-FG}$  and $y$   such that   $dy=*dx$ where we assume orientation of $\Sigma$ such that
  $*d\xi=\sqrt{-\frac{F(\xi)}{G(\eta)}}d\eta, *d\eta=-\sqrt{-\frac {G(\eta)}{F(\xi)}}d\xi$. Then   $y=\frac c2\xi+\frac a2\eta+e$  where  $e=\frac{cb}{2a}+\frac{ad}{2c}$  we choose in such a way that  $(\xi_0,\eta_0)$ has new coordinates $(0,0)$.  If  $z=x+iy$,    then  $z=i\frac{ac}8\widetilde{z}^2$.   Hence  $dx^2+dy^2=a^2c^2\frac{|\widetilde{z}|^2}{16}(d\widetilde{x}^2+d\widetilde{y}^2)$.  Note that in [W] the orientation of $\Sigma$ is the same as in our paper and it should be $dy=J_{\Sigma}dx=*dx$ where the orientation form is $\om_{\Sigma}(X,Y)=g(J_{\Sigma}X,Y)=\frac1{\sqrt{\mathcal V}}d\psi_1\wedge d\psi_2$.

Consequently    $g_{\Sigma}=\frac{16(\xi-\eta)}{a^2c^2|\widetilde{z}|^2}(dx^2+dy^2)=\frac{4(\xi-\eta)}{c^2F-a^2G}(dx^2+dy^2)$.

   Note that $\sqrt{x^2+y^2}=(\frac c2\xi-\frac a2\eta+\gamma)$  where    $\gamma=\frac{bc^2-da^2}{2ac}$.  Define a Lie subalgebra $\mathfrak{h}=\operatorname{span}\{X_1,X_2\}\subset\mathfrak{hol}(M)$ of the Lie algebra of holomorphic Killing  vector fields on $M$.  If we change a basis of this algebra   $Y_1=AX_1+BX_2,Y_2=CX_1+DX_2$,   then for potentials   $\psi_1,\psi_2$  of the fields $Y_1,Y_2$   we get
$\psi_1=A\phi_1+B\phi_2,\psi_2=C\phi_1+D\phi_2$ (up to a constant).  We also have  $|det(Y_1,Y_2)|=|AD-BC||det(X_1,X_2)|$.   Hence the appropriate volumetric coordinates satisfy  $z_1=|AD-BC|z=tz$.

Note that

\begin{align*}\xi+\eta+\frac ba+\frac dc=\frac1c(y+\sqrt{x^2+y^2})-\frac1a(-y+\sqrt{x^2+y^2}), \\ \xi+\frac ba=\frac1c(y+\sqrt{x^2+y^2}), \eta+\frac dc=-\frac1a(-y+\sqrt{x^2+y^2}),\end{align*}
and  $$-x^2=ac\xi\eta+da\xi+cb\eta+db.$$
Now we find how to change the basis to obtain the moment mappings from the Weber work.

We have $$\psi_2=\frac1{\sqrt2}(y'+\sqrt{(x')^2+(y')^2})+B(x')^2,\psi_1=\frac1{\sqrt2}(-y'+\sqrt{(x')^2+(y')^2})+A(x')^2$$ where $\alpha=2A, \beta=2B$ in  new volumetric coordinates corresponding to  $\psi_1,\psi_2$.

Hence   $x'=t x,  y'=t y$ where $t\in\mathbb R, t>0$ is a constant.

Note that $\psi_2=\frac1{\sqrt2}(y'+\sqrt{(x')^2+(y')^2})+B(x')^2=t\frac1{\sqrt2}(y+\sqrt{x^2+y^2})+t^2Bx^2=  t \frac c{\sqrt{2}}\xi+t\frac{bc}{\sqrt{2}a}-t^2Bac\xi\eta-t^2Bda\xi-t^2Bbc\eta-t^2Bdb= (\frac1{\sqrt{2}}t c-t^2Bda)\xi-t^2Bbc\eta-t^2Bac\xi\eta+t\frac{bc}{\sqrt{2}a}-t^2Bdb$  has to be a linear combination of $\xi+\eta,\xi\eta$ and a constant.

Hence    $B=\frac c{t\sqrt{2}(da-bc)}$.

We have $\psi_1=\frac1{\sqrt2}(-y'+\sqrt{(x')^2+(y')^2})+A(x')^2=t\frac1{\sqrt2}(-y+\sqrt{x^2+y^2})+t^2Ax^2=  -t \frac a{\sqrt{2}}\eta-t\frac{da}{\sqrt{2}c}-t^2Aac\xi\eta-t^2Ada\xi-t^2Abc\eta-t^2Adb$.   Similarly for $\psi_1$ we have  $A=\frac a{t\sqrt{2}(da-bc)}$.   Consequently $\alpha=2A=\frac {2a}{t\sqrt{2}(da-bc)}, \beta=2B=\frac {2c}{t\sqrt{2}(da-bc)}$,

$\psi_1=-\frac{t da^2}{\sqrt{2}(da-bc)}\phi_1-\frac{t a^2c}{\sqrt{2}(da-bc)}\phi_2+C_2$

$\psi_2=-\frac{t bc^2}{\sqrt{2}(da-bc)}\phi_1-\frac{t ac^2}{\sqrt{2}(da-bc)}\phi_2+C_1,$  where $C_1,C_2$ are constant.

Hence the area of a parallelogram is   $t\mathcal V=\frac{t^2a^2c^2}{2(da-bc)}\mathcal V$. Consequently  $t=\frac{2(da-bc)}{a^2c^2}$.

To show that generalized  Taub-NUT  and orthotoric  with $F,G$ are isometric we have to prove that

$$\frac{4(\xi-\eta)}{c^2F-a^2G}=\frac{t(1+2t M(ky+\sqrt{x^2+y^2}))}{\sqrt{x^2+y^2}}.$$  We prove that  $k=\frac{a-c}{a+c}$  and  $M=\frac{(c+a)a^2c^2}{4(da-bc)^2}$.
Note that   $c^2F-a^2G=2ac\sqrt{x^2+y^2}$.   Hence we have to show that  $4(\xi-\eta)=2act(1+2t M(ky+\sqrt{x^2+y^2}))$.

\begin{align*} ky+\sqrt{x^2+y^2}=k(\frac a2\eta+\frac c2\xi+e)-\frac a2\eta+\frac c2\xi+\gamma=\\(k-1)\frac a2\eta+(k+1)\frac c2\xi+ke+\gamma\end{align*}  and  $(1-k)\frac a2=(k+1)\frac c2$   thus  $k=\frac{a-c}{a+c}$ and
$t^2 ac^2 M \frac{2a}{c+a}=2$  and   $M \frac{4(da-bc)^2}{a^2c^2(c+a)}=1$.  It implies   $M=\frac{(c+a)a^2c^2}{4(da-bc)^2}$.  Then   $2t M(ke+\gamma)=-1$.
  \end{proof}
Now we introduce global coordinates on  $M=\mathbb C^2$. Let us introduce coordinates  (compare [W])
$u=\frac12\sqrt{M t c a}\widetilde{x}, v=-\frac12\sqrt{M t c a}\widetilde{y}$.  Then $$ \xi=\frac{u^2}{Mt c}-\frac ba,  \eta=-\frac{v^2}{Mt a}-\frac dc,$$  $$ Md\xi=\frac{2udu}{t c},  Md\eta=-\frac{2vdv}{t a}.$$    We also have
$g_{\Sigma}=\frac 2M (1+(1+k)u^2+(1-k)v^2)(du^2+dv^2)$  and change the coordinates  $\theta_i$   by   $\theta_i'=\frac1{\sqrt{2}}\theta_i$.    Then   (we denote  $\theta_i'$ back by  $\theta_i$)
$g=g_{\Sigma}+G^{ij}d\theta_id\theta_j$,   where

\begin{align*} MG^{11}=\frac{2v^2\left((1+(1+k)u^2)^2+(1+k)^2u^2v^2\right)}{1+(1+k)u^2+(1-k)v^2},\\  MG^{12}=MG^{21}=\frac{2u^2v^2\left(2+(1-k^2)(u^2+v^2)\right)}{1+(1+k)u^2+(1-k)v^2},\\
MG^{22}=\frac{2u^2\left((1+(1-k)v^2)^2+(1-k)^2u^2v^2\right)}{1+(1+k)u^2+(1-k)v^2}.\end{align*}

Let us introduce  new coordinates  $$x_1=v\cos\theta_1,  y_1=v\sin\theta_1,  x_2=u\cos\theta_2,  y_1=u\sin\theta_2.$$    We get

$\frac{\p}{\p x_1}=\cos\theta_1\frac{\p}{\p v}-\frac1v\sin\theta_1\frac{\p}{\p\theta_1},  \frac{\p}{\p y_1}=\sin\theta_1\frac{\p}{\p v}+\frac1v\cos\theta_1\frac{\p}{\p\theta_1}$,

$\frac{\p}{\p x_2}=\cos\theta_2\frac{\p}{\p u}-\frac1u\sin\theta_2\frac{\p}{\p\theta_2},  \frac{\p}{\p y_2}=\sin\theta_2\frac{\p}{\p u}+\frac1u\cos\theta_2\frac{\p}{\p\theta_2}$.

Consequently

$\frac M2 g(\frac{\p}{\p x_1},\frac{\p}{\p x_1})=(1-k)x_1^2+\frac{(1+(1+k)u^2)^2+(k-1)x_1^2(1+(1+k)u^2)+y_1^2(1+k)^2u^2}{1+(1+k)u^2+(1-k)v^2}$,

$ \frac M2g(\frac{\p}{\p x_1},\frac{\p}{\p y_1})=(1-k)x_1y_1-\frac{(1+k)^2u^2x_1y_1}{1+(1+k)u^2+(1-k)v^2}+\frac{(1+(1+k)u^2)(1-k)x_1y_1}{1+(1+k)u^2+(1-k)v^2}$,

$\frac M2 g(\frac{\p}{\p y_1},\frac{\p}{\p y_1})=(1-k)y_1^2+\frac{(1+(1+k)u^2)^2+(1+(1+k)u^2)(k-1)y_1^2+x_1^2(1+k)^2u^2}{1+(1+k)u^2+(1-k)v^2}$,

$\frac M2 g(\frac{\p}{\p x_1},\frac{\p}{\p x_2})=y_1y_2\frac{2+(1-k^2)(u^2+v^2)}{1+(1+k)u^2+(1-k)v^2}$,

$\frac M2 g(\frac{\p}{\p x_1},\frac{\p}{\p y_2})=-y_1x_2\frac{2+(1-k^2)(u^2+v^2)}{1+(1+k)u^2+(1-k)v^2}$,

$\frac M2 g(\frac{\p}{\p x_2},\frac{\p}{\p x_2})=(1+k)x_2^2+\frac{(1+(1-k)v^2)^2+(-(k+1)x_2^2(1+(1-k)v^2)+y_2^2(1-k)^2v^2}{1+(1+k)u^2+(1-k)v^2}$,

$ \frac M2 g(\frac{\p}{\p x_2},\frac{\p}{\p y_2})=(1+k)x_2y_2-\frac{(1-k)^2v^2x_2y_2}{1+(1-k)v^2+(1+k)u^2}+\frac{(1+(1-k)v^2)(1+k)x_2y_2}{1+(1+k)u^2+(1-k)v^2}$,

$\frac M2 g(\frac{\p}{\p y_2},\frac{\p}{\p y_2})=(1+k)y_2^2+\frac{(1+(1-k)v^2)^2+(1+(1-k)v^2)(-k-1)y_2^2+x_2^2(1-k)^2v^2}{1+(1+k)u^2+(1-k)v^2}$,

$\frac M2 g(\frac{\p}{\p y_1},\frac{\p}{\p y_2})=x_1x_2\frac{2+(1-k^2)(u^2+v^2)}{1+(1+k)u^2+(1-k)v^2}$,

$\frac M2 g(\frac{\p}{\p y_1},\frac{\p}{\p x_2})=-x_1y_2\frac{2+(1-k^2)(u^2+v^2)}{1+(1+k)u^2+(1-k)v^2}$.

Since   $u^2=x_2^2+y_2^2,  v^2=x_1^2+y_2^2$,   it follows  $g$  is defined in coordinates  $z_1=x_1+iy_2,z_2=x_2+iy_2$ on the whole of   $\mathbb C^2$.

  On $int(\Sigma)\times T^2$  we have coordinates   $u,v,\theta_1,\theta_2$  where  $u>,v>0,  \theta_1\in [0,2\pi),\theta_2\in[0,2\pi)$  (in fact  $\theta_i$  are coordinates on  $T^2=\mathbb R^2\slash 2\pi\mathbb Z^2$).  It is easy to check that the maps   $u,v,\theta_1,\theta_2$ and  $x_1,y_1,x_2,y_2$   agree on $int(\Sigma)\times T^2$  hence for  $u>0,v>0$   and   $int(\Sigma)\times T^2$ in  new coordinates $z_1=x_1+iy_2,z_2=x_2+iy_2$ is $(\mathbb C-\{0\})\times(\mathbb C-\{0\})$.
The basis of Killing vector fields is $Y_1= x_1\frac{\p}{\p y_1}-y_1\frac{\p}{\p x_1}, Y_2= x_2\frac{\p}{\p y_2}-y_2\frac{\p}{\p x_2}$. In  the  open set  $\{\xi>\xi_0,\eta<\eta_0$  the  distributions $\operatorname{ker} (IJ-I),\operatorname{ker}(IJ+I)$ are disjoint  from $Null I$ and there our surface is orthotoric.  Now we want to see what happens if  $\xi=\xi_0$ or  $\eta=\eta_0$. We will show that then one of distributions $\operatorname{ker} (IJ-I),\operatorname{ker}(IJ+I)$ coincides with $Null I$.  The (real) surface  $L=\{Y_1=0\}$   in  new coordinates is the plane
$x_1=y_1=0$.    We shall show that it coincides with one of the distributions  $\mathcal D,  \mathcal E$  from the definition of  $QCH$  K\"ahler surfaces.  Let  $TL$ be a tangent space of  $L$.

Let   $T=\n Y_1,  S=\n Y_2$.   Since on   $L$  we have  $0=(L_{Y_1}-\n_{Y_1})I=[T,I]$   and  $[T,J]=0$,  it follows that    $\operatorname{ker} T=TL$  is   $I,J$  invariant.    The equality   $dim  TL=2$   implies
$I=J$  on  $TL$  or  $I=-J$  on  $TL$.  Hence   $TL$   coincides with  $\operatorname{ker}(IJ-id)$   or   $\operatorname{ker}(IJ+id)$  hence with  $\mathcal D$  or  $\mathcal E$.  Note that   $R(Y_1,Y_2)=[T,S]$.  Consequently  on $TL$
  we have$[T,S]=0$.  It implies   $S(TL)\subset  TL$.   Thus   $S=a I$  on  $TL$   which means   $\n_{Y_2}I=[S,I]=0$  on   $L$.  Therefore    $Y_2, JY_2\in  TL$  and   $Null(I)=\operatorname{span}(Y_2, JY_2)=TL$ where  $Null(I)=\{X:\n_XI=0\}$.   Note that at a point    $(0,0,0,0)$   we have  $\n I=0$,   since    $\|\n (\xi-\eta)\|^2= \frac1{\xi-\eta}(F(\xi)-G(\eta)$  tends to $0$, when  $\xi\rightarrow \xi_0, \eta\rightarrow  \eta_0$.

We find  the K\"ahler forms $\omega_J,\omega_I$.  We have  (where again   $\theta_i'=\theta_i$)

$\omega_J=\sqrt{2}(d\phi_1\wedge d\theta_1+d\phi_2\wedge d\theta_2)$   where   $\phi_1=\frac{v^2}{\sqrt{2}M}(1+(1+k)u^2),  \phi_2=\frac{u^2}{\sqrt{2}M}(1+(1-k)v^2)$. It is easy to show

$\frac M2\omega_J=(1+(1+k)u^2)dx_1\wedge dy_1+(1+(1-k)v^2)dx_2\wedge dy_2+(1+k)(x_2dx_2+y_2dy_2)\wedge(x_1dy_1-y_1dx_1)+(1-k)(x_1dx_1+y_1dy_1)\wedge(x_2dy_2-y_2dx_2)$.

Note that  $\om_I=d\xi\wedge(dt+\eta dz)-d\eta\wedge(dt+\xi dz)$  (see [A-C-G]).  Hence

$\frac M2\omega_I=-(1+(1+k)u^2)dx_1\wedge dy_1+(1+(1-k)v^2)dx_2\wedge dy_2+(1+k)(x_2dx_2+y_2dy_2)\wedge(x_1dy_1-y_1dx_1)-(1-k)(x_1dx_1+y_1dy_1)\wedge(x_2dy_2-y_2dx_2)$.
\section{  Exceptional generalized  Taub-NUT.}

\begin{thm}  The exceptional generalized Taub-NUT  surface  is a QCH K\"ahler surface of Calabi type.\end{thm}

\begin{proof}  In an exceptional generalized  Taub-NUT corresponding  to  $k=1$  we have   $g_{\Sigma}= 2(1+2u^2)(du^2+dv^2)$ and  $g=g_{\Sigma}+G^{ij}d\theta_id\theta_j$,   where

$MG^{11}=\frac{v^2((1+2u^2)^2+4u^2v^2}{1+2u^2},  MG^{12}=MG^{21}=\frac{2u^2v^2(u^2+v^2)}{1+2u^2}$,

$MG^{22}=\frac{u^2}{1+2u^2}$.
We also have  $\omega_J=\frac{2v}{\sqrt{2}}(1+2u^2)dv\wedge d\theta_1+\frac{4uv^2}{\sqrt{2}}du\wedge d\theta_1+\frac{2u}{\sqrt{2}}du\wedge d\theta_2$,
 $\omega_I=\frac{2v}{\sqrt{2}}(1+2u^2)dv\wedge d\theta_1-\frac{4uv^2}{\sqrt{2}}du\wedge d\theta_1-\frac{2u}{\sqrt{2}}du\wedge d\theta_2$,

$I\frac{\p}{\p u}=-\sqrt{2}\frac{(1+2u^2)}{u}\frac{\p}{\p\theta_2},   I\frac{\p}{\p v}=-2\sqrt{2}v\frac{\p}{\p\theta_2}+\frac{\sqrt{2}}v\frac{\p}{\p\theta_1}$

and
$I\frac{\p}{\p\theta_2}=\frac1{\sqrt{2}}\frac{u}{1+2u^2}\frac{\p}{\p u},  I\frac{\p}{\p\theta_1}=\sqrt{2}\frac {u^2}{1+2u^2}\frac{\p}{\p u}-\frac{\sqrt{2}}2v\frac{\p}{\p v}$.

One can easily  check that   $N_I(X,Y)=0$, where $N_I$ is the Nijenhuis tensor of the almost Hermitian structure $I$.   Hence  $I$ is integrable.
We have  $\frac12(\omega_J-\omega_I)=\frac{4uv^2}{\sqrt{2}}du\wedge d\theta_1+\frac{2u}{\sqrt{2}}du\wedge d\theta_2, \phi=\frac12(\omega_J+\omega_I)=\frac{2v}{\sqrt{2}}(1+2u^2)dv\wedge d\theta_1$.

Consequently $d\phi=\frac{8vu}{\sqrt{2}}du\wedge dv\wedge d\theta_1=\frac{4u}{1+2u^2}du\wedge\phi=d\ln(1+2u^2)\wedge\phi$ and $(M,g,I)$  is conformally K\"ahler.     It follows that exceptional Taub-NUT is of Calabi type.\end{proof}

\section{   Exceptional half plane.}   Let  $(M,g,J)$ be an exceptional half plane instanton. Then we  have $M=\mathbb C^2$.  Recall that  $g=(1+x^2)(dx^2+dy^2)+G^{ij}d\theta_id\theta_j$,  where    $G^{11}=\frac{x^2}{1+x^2}, G^{12}=G^{21}=\frac{2x^2y}{1+x^2},  G^{22}=\frac{(1+x^2)^2+4x^2y^2}{1+x^2}$  (see [W]).

\begin{thm}  An exceptional  half plane  is a QCH K\"ahler surface of  Calabi  type.\end{thm}

\begin{proof}   First we introduce global coordinates on  $M$. These new  coordinates are $x_1=x\cos\theta_1,  y_1=x\sin\theta_1,  x_2=y,  y_2=\theta_2$.
Note that
$\frac{\p}{\p x_1}=\cos\theta_1\frac{\p}{\p x}-\frac1x\sin\theta_1\frac{\p}{\p\theta_1},  \frac{\p}{\p x_1}=\frac{\p}{\p y},  \frac{\p}{\p y_1}=\frac{\p}{\p\theta_2}  $.
It follows that

$g(\frac{\p}{\p x_1},\frac{\p}{\p x_1})=1+x_1^2-\frac{y_1^2}{1+x^2},  g(\frac{\p}{\p x_1},\frac{\p}{\p y_1})=x_1y_1+\frac{x_1y_1}{1+x^2}, g(\frac{\p}{\p x_1},\frac{\p}{\p x_2})=0$,

$g(\frac{\p}{\p x_1},\frac{\p}{\p y_2}) =-\frac{2y_1x_2}{1+x^2}, g(\frac{\p}{\p x_2},\frac{\p}{\p x_2})=1+x^2, g(\frac{\p}{\p y_2},\frac{\p}{\p y_2})=\frac{(1+x^2)^2+4x^2x_2^2}{1+x^2}$.

The K\"ahler form  of $(M,g,I)$ is     $\omega_I=\frac{\sqrt{2}}{|Ric|}(-d\mathcal R^1\wedge d\theta_1-d\mathcal R^2\wedge d\theta_2)=xdx\wedge d\theta_1-(1+x^2)dy\wedge d\theta_2+2xydx\wedge d\theta_2$, where $|Ric|=\sqrt{g(Ric,Ric)}$ and $\mathcal R^2=\frac{2y}{1+x^2},\mathcal R^1=\frac1{1+x^2}$ are Ricci potentials  (see [W]).
Hence

$I\frac{\p}{\p x}=\frac{(1+x^2)}{x}\frac{\p}{\p\theta_1},   I\frac{\p}{\p y}=-\frac{\p}{\p\theta_2}+2y\frac{\p}{\p\theta_1}$

and

$I\frac{\p}{\p\theta_2}=-\frac{2xy}{1+x^2}\frac{\p}{\p x}+\frac{\p}{\p y},  I\frac{\p}{\p\theta_1}=-\frac x{1+x^2}\frac{\p}{\p x}$.

 The  K\"ahler form   of  $(M,g,J)$ is $\omega_J=xdx\wedge d\theta_1+(1+x^2)dy\wedge d\theta_2+2xydx\wedge d\theta_2$.  Since    $d\omega_I=\frac{4xdx}{1+x^2}\wedge\omega_I$,   it follows that the Lee form for  $(M,g,I)$ is $\theta=d\ln(1+x^2)$, i.e. $d\omega_I=2\theta\wedge\omega_I$.
One can easily  check that   $N_I(X,Y)=0$, where $N_I$ is the Nijenhuis tensor of the almost Hermitian structure $I$.  Hence  $I$ is integrable  and $(M,g,I)$  is conformally K\"ahler.  Thus  $(M,g,J)$  is a K\"ahler surface of Calabi type. The Killing fields are  $Y_1= x_1\frac{\p}{\p y_1}-y_1\frac{\p}{\p x_1}, Y_2= \frac{\p}{\p y_2}$.

Note that on  $M$ acts a cylinder  $\mathbb R\times S^1$  rather than   a torus $T^2$.\end{proof}

\section{  The  Burns  metric.}  Burns proved that the metric  $\omega=-i\p\bp (||z||^2+m\ln ||z||^2)$ on  $\mathbb C^2-\{0\}$  extends to the K\"ahler metric on the blow-up of $\mathbb C^2$  at $0$. Let us introduce the coordinates $(u,z)$ on  $\widehat{\mathbb C^2}$ such that for standard coordinates $z_1,z_2$  on $\mathbb C^2$  we have $z_1=z,z_2=uz$ or in the second map $z_1=uz,z_2=z$ on $\mathbb C^2-\{0\}$. Then  $||z||^2=|z_1|^2+|z_2|^2=|z|^2(1+|u|^2)$ and consequently

$\p\bp (||z||^2)=\p\bp (|z|^2(1+|u|^2)=\p(zd\overline{z}+|u|^2zd\overline{z}+|z|^2ud\overline{u}=dz\wedge d\overline{z}+|u|^2dz\wedge d\overline{z}+\overline{u}zdu\wedge d\overline{z} +    |z|^2du\wedge d\overline{u}+\overline{z}udz\wedge d\overline{u}$,

$\p\bp (\ln ||z||^2)=\p\bp (\ln |z|^2(1+|u|^2)=\p\bp (\ln(1+|u|^2)=\p(\frac{ud\overline{u}}{1+|u|^2})=\frac{du\wedge d\overline{u}}{1+|u|^2}-\frac{u \overline{u} du\wedge d\overline{u}}{(1+|u|^2)^2}=
\frac{ du\wedge d\overline{u}}{(1+|u|^2)^2}$.

\begin{thm}  The space $\widehat{\mathbb C^2}$ with the Burns metric  is a QCH K\"ahler surface of  Calabi  type.\end{thm}

\begin{proof} We have $\omega=-i((1+|u|^2)dz\wedge d\overline{z}+(|z|^2+\frac m{(1+|u|^2)^2})du\wedge d\overline{u}+\overline{u}zdu\wedge d\overline{z}+u\overline{z}dz\wedge d\overline{u})$.

One can also see that   $\omega^2=-2(|z|^2+\frac m{1+|u|^2})dz\wedge d\overline{z}\wedge du\wedge d\overline{u})$   and the Ricci form is

$\rho=-i\p\bp\ln(|z|^2+\frac m{1+|u|^2})$.

Note that $\bp\ln(|z|^2+\frac m{1+|u|^2})=\frac1{|z|^2+\frac m{1+|u|^2}}(zd\overline{z}-\frac{mud\overline{u}}{(1+|u|^2)^2})$.

We also  have
\begin{align*} \p\bp\ln(|z|^2+\frac m{1+|u|^2})=\frac1{|z|^2+\frac m{1+|u|^2}}(dz\wedge d\overline{z}-\frac{mdu\wedge d\overline{u}}{(1+|u|^2)^2}+\frac{2m|u|^2du\wedge d\overline{u}}{(1+|u|^2)^3})
\\ -\frac1{(|z|^2+\frac m{1+|u|^2})^2}(\overline{z}dz- \frac{m\overline{u}du}{(1+|u|^2)^2})\wedge (zd\overline{z}-\frac{mud\overline{u}}{(1+|u|^2)^2}).\end{align*}

   Consequently \begin{align*}\rho=\frac {-im}{(|z|^2+|z|^2|u|^2+m)^2}((1+|u|^2)dz\wedge d\overline{z}+\\\frac{(|z|^2+|z|^2|u|^2+m)(|u|^2-1)-m|u|^2}{(1+|u|^2)^2}du\wedge d\overline{u}+\overline{u}zdu\wedge d\overline{z}+u\overline{z}dz\wedge d\overline{u}).\end{align*}

    Thus  \begin{align*}\omega\wedge \rho= -\frac m{(|z|^2+|z|^2|u|^2+m)^2}(\frac{(|z|^2+|z|^2|u|^2+m)(|u|^2-1)-m|u|^2}{(1+|u|^2)}+\\(|z|^2+\frac m{(1+|u|^2)^2})(1+|u|^2)-2|u|^2|z|^2)du\wedge d \overline{u}\wedge dz\wedge d\overline{z}=0.\end{align*}

The diffeomorphism    $\phi:\mathbb C^2-\{0\}\rightarrow \mathbb C^2-\{0\}$,  $\phi(z)=\frac z{||z||^2}$, extends to the diffeomorphism   $\phi:\widehat{\mathbb C^2}\rightarrow\mathbb {CP}^2-\{[1,0,0]\}$( see [L1]).   Note that in the first map
$\phi(u,z)=[|z|^2(1+|u|^2),z,uz]=[\overline{z}(1+|u|^2),1,u]$   and in the second map  $\phi(u,z)=[|z|^2(1+|u|^2),uz,z]=[\overline{z}(1+|u|^2),u,1]$.   These are smooth mappings and $\phi((0,[u_1,u_2]))=[0,u_1,u_2]$.  Hence it is a diffeomorphism. We shall identify  $\widehat{\mathbb C^2}$ with its image in  $\mathbb{CP}^2$ with  complex structure and the metric  implanted by $\phi$.

If  $z_1,z_2$ are this time the coordinates on  $\mathbb{CP}^2$ then  $z_1=\overline{z}(1+|u|^2),z_2=u$   and  $dz_1=d\overline{z}(1+|u|^2)+\overline{z}(ud\overline{u}+\overline{u}du), dz_2=du$.  Note that   $u=z_2,  z=\frac{\overline{z_1}}{1+|z_2|^2}$.  Consequently

$du=dz_2, dz=\frac{\overline{dz_1}}{1+|z_2|^2}-\overline{z_1}\frac{z_2d\overline{z_2}+\overline{z_2}dz_2}{(1+|z_2|^2)^2}.$

 $\omega=-i(((1+|u|^2)\frac{\overline{dz_1}}{1+|z_2|^2}-\overline{z_1}\frac{z_2d\overline{z_2}+\overline{z_2}dz_2}{(1+|z_2|^2)^2})\wedge (\frac{dz_1}{1+|z_2|^2}-z_1\frac{\overline{z_2}dz_2+z_2d\overline{z_2}}{(1+|z_2|^2)^2})+(|\frac{\overline{z_1}}{1+|z_2|^2}|^2+\frac m{(1+|z_2|^2)^2})dz_2\wedge d\overline{z_2}+\overline{z_2}\frac{\overline{z_1}}{1+|z_2|^2}dz_2\wedge (\frac{dz_1}{1+|z_2|^2}-z_1\frac{\overline{z_2}dz_2+z_2d\overline{z_2}}{(1+|z_2|^2)^2})+z_2\frac{z_1}{1+|z_2|^2}(\frac{\overline{dz_1}}{1+|z_2|^2}-\overline{z_1}\frac{z_2d\overline{z_2}+\overline{z_2}dz_2}{(1+|z_2|^2)^2}\wedge d\overline{z_2}))$

   Hence $\omega(I,I)=\omega$,  similarly   $\rho(I,I)=\rho$.    It follows that  $\omega,\rho$ are $I,J$ invariant.    Consequently    $\frac{\sqrt{2}}{|\rho|}\rho=\omega_I$   where  $I$  is the complex structure of  $\mathbb{CP}^2$.

Note that   $\rho\wedge\rho=-2(\frac m{(|z|^2+|z|^2|u|^2+m)^2})^2vol$   which implies

 $\omega_I=-i((1+|u|^2)dz\wedge d\overline{z}+\frac{(|z|^2+|z|^2|u|^2+m)(|u|^2-1)-m|u|^2}{(1+|u|^2)^2}du\wedge d\overline{u}+\overline{u}zdu\wedge d\overline{z}+u\overline{z}dz\wedge d\overline{u})$.

 We also have  $\frac12d(\omega_J-\omega_I)=d\ln(|z|^2+|z|^2|u|^2+m)\wedge\frac12(\omega_J-\omega_I)=d\ln(||z||^2+m)\wedge\frac12(\omega_J-\omega_I)$.    Hence the Burns metric is of Calabi type.\end{proof}

 \vskip1cm
 {\bf Acknowledgements }    The author thanks the reviewer for all his remarks which improved the paper.
\vskip1cm
\centerline{\bf References.}
\par
\medskip
[A-C-G] V. Apostolov, D. Calderbank and P. Gauduchon {\it The geometry of weakly self-dual K\"ahler surfaces}, Compositio Mathematica 135,No.3 (2003), 279-322.
\par
\medskip
[A-G] V. Apostolov, P. Gauduchon, { \it The Riemannian Goldberg-Sachs Theorem}, Internat. J. Math. 8 (1997), 421-439.
\par
\medskip
[B] D. Burns, {\it  Twistors and harmonic maps}, Lecture, Amer. Math. Soc. Conference,  Charlotte, NC, October 1986
\par
\medskip
[D] A. Derdzi\'nski, {\it  Self-dual K\"ahler manifolds and Einstein manifolds of dimension four}, Compos. Math. 49, (1983), 405-433.
\par
\medskip
[Do] S. Donaldson, {\it Constant scalar curvature metrics on toric surfaces}, Geometric and Functional Analysis, 19, No.1, (2009), 83-136
\par
\medskip
[D-T]   M. Dunajski and P. Tod, {\it  Four-dimensional metrics conformal to K\"ahler },
Mathematical Proceedings of the Cambridge Philosophical Society, Volume 148, (2010), 485-503.
\par
\medskip
[G-M] G. Ganchev and V. Mihova,{\it K\"ahler manifolds of quasi-constant holomorphic sectional curvature}, Central European Journal of Mathematics, 6, (2008),43-75.
\par
\medskip
[G-R] G.Gibbons, P.Ruback,   {\it The hidden symmetries of multi-centre metrics} Comm.  Math.  Phys. 115,  (1988), 267-300.

[G] A. Gray,  {\it Curvature identities for Hermitian
and almost Hermitian manifolds}, Tohoku Math. Journ,
28 (1976), 601-612.
\par
\medskip
[J-1] W. Jelonek, {\it K\"ahler surfaces with quasi constant holomorphic curvature}, Glasgow Math. J. 58, (2016), 503-512.
\par
\medskip
[J-2] W. Jelonek, {\it Semi-symmetric  K\"ahler surfaces}, Colloq. Math.  148, (2017), 1-12.
\par
\medskip
[J-3] W. Jelonek, {\it Complex foliations and  K\"ahler QCH surfaces},  Colloq. Math.  156, (2019), 229-242.
\par
\medskip
[J-4] W. Jelonek, {\it Einstein Hermitian and anti-Hermitian 4-manifolds}, Ann.  Po\-lon. Math. 81, (2003), 7-24.
\par
[J-5] W. Jelonek, {\it  QCH  K\"ahler surfaces II},  Journal of Geometry and Physics, (2020), 103735.
\par
\medskip
[J-6] W. Jelonek, {\it Generalized  orthotoric   K\"ahler surfaces},  Ann. Polon.  Math., (2022), 193-120.
\par
\medskip
[J-7] W.Jelonek, {\it Compact pseudosymmetric K\"ahler Manifolds}, Colloq.  Math, 117.2,   (2009), 243-249.
\par
\medskip
[J-8] W.Jelonek,  {\it K\"ahler manifolds with quasi-constant holomorphic curvature},  Ann. of Global Analysis and Geometry, 36, No.2, (2009)
\par
\medskip
[J-M] W. Jelonek and E. Mulawa, {\it Generalized  Calabi type  K\"ahler surfaces},  Journal of Geometry and Physics, 182,(2022)

\par
\medskip
[L]  C. Lebrun,  {\it  Explicit self-dual metrics on  $\mathbb{CP}^2\sharp...\sharp\mathbb{CP}^2$},  J.  Diff.  Geom.  34, (1991), 223-253.
\par
\medskip
[L1]  C. Lebrun,  {\it Poon's self-dual metrics and K\"ahler geometry }, J. Diff. Geom, 28, (1988), 341-343.
\par
\medskip
[M] E. Mulawa, {\it A detailed discription of generalized Calabi type K\"ahler surfaces}, Journal of Geometry and Physics 198, (2024), 105112
\par
\medskip
[N] T. Noda,  {\it  A special Lagrangian fibration in the Taub-NUT space}, J.Math. Soc. of Japan, vol.60, No.3, (2008), 653-663.
\par
\medskip
[O] Z. Olszak, {\it Bochner flat K\"ahlerian manifolds with a certain condition on the Ricci tensor}, Simon Stevin, 63, (1989), 295-303
\par
\medskip
[W] B. Weber, {\it Generalized K\"ahler Taub-NUT metrics and two exceptional instantons}  Communications in Analysis and Geometry, 30, No.7, (2022), 1575-1632.

\end{document}